\documentclass[11pt]{amsart}

\usepackage[OT1]{fontenc}
\usepackage[left=3.6cm,top=3.8cm,right=3.6cm]{geometry}
\usepackage{comment}
\usepackage[utf8]{inputenc}
\usepackage{enumerate}
\usepackage{marvosym}
\usepackage{amsthm,amsmath,amssymb,mathrsfs}

\usepackage{verbatim} 
\usepackage{esint}

\numberwithin{equation}{section}

\usepackage{xcolor}
\usepackage[pagebackref]{hyperref}
\hypersetup{
   colorlinks,
    linkcolor={red!60!black},
    citecolor={blue!60!black},
    urlcolor={blue!90!black}
}

\numberwithin{equation}{section}

\theoremstyle{plain}

\newtheorem{theorem}{Theorem}[section]
\newtheorem{lemma}[theorem]{Lemma}
\newtheorem{proposition}[theorem]{Proposition}
\newtheorem{corollary}[theorem]{Corollary}

\theoremstyle{definition}
\newtheorem{definition}[theorem]{Definition}

\newtheorem{remark}[theorem]{Remark}

\theoremstyle{remark}
\numberwithin{equation}{section}

\newcommand{\rr}{\mathbb{R}}

\newcommand{\R}{\mathbb{R}}
\newcommand{\Z}{\mathbb{Z}}

\newcommand{\eps}{\varepsilon}
\newcommand{\norm}[1]{\left\lVert #1 \right\rVert}
\newcommand{\1}{\mathbf{1}}

\def\Xint#1{\mathchoice
  {\XXint\displaystyle\textstyle{#1}}%
  {\XXint\textstyle\scriptstyle{#1}}%
  {\XXint\scriptstyle\scriptscriptstyle{#1}}%
  {\XXint\scriptscriptstyle\scriptscriptstyle{#1}}%
  \!\int}
\def\XXint#1#2#3{{\setbox0=\hbox{$#1{#2#3}{\int}$}
    \vcenter{\hbox{$#2#3$}}\kern-.5\wd0}}

\def\avgint{\Xint-}
\newcommand{\vertiii}[1]{{\left\vert\kern-0.25ex\left\vert\kern-0.25ex\left\vert #1 
    \right\vert\kern-0.25ex\right\vert\kern-0.25ex\right\vert}}

\numberwithin{equation}{section}

\begin{document}

\title[Weighted skeletons]{Weighted estimates for maximal functions associated to skeletons}

\author{Andrea Olivo}
\address{Departamento de Matem\'atica,
Facultad de Ciencias Exactas y Naturales,
Universidad de Buenos Aires, Ciudad Universitaria
Pabell\'on I, Buenos Aires 1428 Capital Federal Argentina} \email{aolivo@dm.uba.ar}

\author{Ezequiel Rela}
\address{Departamento de Matem\'atica,
Facultad de Ciencias Exactas y Naturales,
Universidad de Buenos Aires, Ciudad Universitaria
Pabell\'on I, Buenos Aires 1428 Capital Federal Argentina} \email{erela@dm.uba.ar}

\thanks{ Both authors are partially supported by grant PIP (CONICET) 11220110101018. Second author is also supported by granbt UBACyT 20020170200057BA}



\begin{abstract}
We provide quantitative weighted estimates for the $L^p(w)$ norm of a maximal operator associated to cube skeletons in $\mathbb{R}^n$. The method of proof differs from the usual in the area of weighted inequalities since there are no covering arguments suitable for the geometry of skeletons. We use instead a combinatorial strategy that allows to obtain, after a linearization and discretization, $L^p$ bounds for the maximal operator from an estimate related to intersections between skeletons and $k$-planes.
\end{abstract}

\maketitle

\section{Introduction and main results}\label{sec:intro}

\subsection{Averaging operators and packings} The purpose of this article is to study weighted estimates for a certain type of maximal operators related to a geometric problem of ``packing objects''. The most famous example of this type of problems is the Kakeya needle problem: how small can be a set $E\subset \mathbb{R}^n$ containing a unit line segment in every possible direction $e\in \mathbb{S}^{n-1}$? It is known that those sets can be of null Lebesgue measure, but the question regarding the smallest possible value for the dimension in $\mathbb{R}^n$ remains open for $n\ge 3$. The best known results on Kakeya with respect to the Hausdorff dimension were obtained as a consequence of $L^p$ bounds for an appropriate maximal operator defined by averaging over thin tubes pointing in any possible direction. More precisely, the Kakeya maximal function is defined by
\begin{equation*}
\mathcal{K}_\delta(f)(e)=\sup_{x\in\mathbb{R}^n}\frac{1}{|T_e^\delta(x)|}\int_{T_e^\delta(x)}|f(x)|\ dx,\qquad e\in\mathbb{S}^{n-1},
\end{equation*}
where $T_e^\delta(x)$ is a $1\times \delta$-tube (by this we mean a tube of length 1 and cross section of radius $\delta$) centered at $x$ in the direction of $e\in\mathbb{S}^{n-1}\subset\mathbb{R}^n$. Due to the existence of zero measure Kakeya sets, this operator can only be bounded on $L^p$ with a dependence on $\delta$ that blows up when $\delta\to 0$. It is precisely this rate of blow up the key ingredient to derive dimension bounds for the Kakeya sets.

 Other examples involving packing of circles of every possible radii or centered at any point of a prescribed set lead to the study of the corresponding maximal operators. 
For example, the problem of packing circles is related to the properties of the spherical maximal function
\begin{equation}\label{eq:spherical}
\mathcal{M}_{\text{sph}}f(x)=  \sup\limits_{ r>0} \frac{1}{\sigma(S^{n-1}(x,r))}\int _ {S^{n-1}(x,r)} |f(y)| \,d\sigma,
\end{equation}
where $\sigma$ is the surface measure on the sphere and $S^{n-1}(x,r)$ denotes the $n-1$ sphere centered on $x$ and with radius $r$. Results for the boundedness of this operator can be found in \cite{Stein76} for $n\ge 3$ and in \cite{Bou86} for the more difficult case of $n=2$ (see also \cite{KW99} and \cite{wol97} for related problems involving similar maximal operators).

In this article we focus on the maximal operator that arises when studying configurations of cube skeletons. Roughly speaking, the $k$-dimensional boundary of a cube with axes parallel sides in $\mathbb{R}^n$ for $0\le k < n$. In $\mathbb{R}^3$, for example, the case $k=2$ consists in considering the faces of the cube, $k=1$ is for the set of edges of the cube and the case $k=0$ describes the vertices. The problem of finding minimal values of the size for sets in $\mathbb{R}^n$ containing $k$-skeletons centered at any point was recently studied by Keleti, Nagy and Shmerkin in \cite{KNS18} and also by Thornton in \cite{Tho17}. The main result in \cite{KNS18} is that there exists a set containing a 1-skeleton centered at any point of the unit square $[0,1]^2$ having Lebesgue measure zero and, moreover, with zero Hausdorff dimension. The next natural step is to study how small can the dimension of such sets be for different notions of dimension. In \cite{OS18} Shmerkin and the first author study the boundedness properties of the corresponding maximal averaging operator. 

The naive analog of \eqref{eq:spherical} for skeletons replacing the spherical surface by the surface of the cube is, as it was observed in \cite{KNS18}, inadequate. The operator introduced in \cite{OS18} is the following version with a restricted range on the radius and averaging over a fattened $k$-skeleton.

\begin{definition}
Let $\delta>0$ and $0\le k < n$. For a given function $f\in L^1_{loc}(\mathbb{R}^n)$, the $k$-skeleton maximal function is defined by
\begin{equation}\label{eq:SkeletonMaximal}
M^{k}_{\delta}f(x) = \sup\limits_{1 \leq r \leq2} \min\limits_{j=1}^{N}  \displaystyle\frac{1}{|S_{k,\delta}^{j}(x,r)|} \int _ {S_{k,\delta}^{j}(x,r)} |f(y)|\, dy.
\end{equation}
Here $S_{k,\delta}^{j}(x,r)$ denotes, for $j=1,\ldots, \binom{n}{k}2^{n-k}$, a $\delta$-neighborhood of each one of the $k$-faces of the $k$-skeleton centered at $x$ with radius $r$.
\end{definition}

As in the case of the Kakeya problem, due to the example from \cite{KNS18} cited above, this operator is bounded on $L^p$  (this is a consequence of the boundedness properties of the classical Hardy-Littlewood maximal operator) but with a blow up when $\delta\to 0$. In this direction we mention the following result from \cite{OS18}.
\begin{theorem}[Olivo-Shmerkin, 2018]\label{thm:main-OS18}
Given $0\le k< n$, $1\le p <\infty$ and $\eps>0$, there exist  positive constants $C'(n,k,\eps), C(n,k)$ such that
\begin{equation*}
C'(n,k,\eps)\cdot \delta^{\frac{k-n}{2np}+\eps} \leq \norm{M^{k}_{\delta}}_{L^p(\mathbb{R}^n)\to L^p(\mathbb{R}^n)} \leq C(n,k)\cdot \delta^{\frac{k-n}{2np}},
\end{equation*}
for all $\delta\in (0,1)$.
\end{theorem}

From this theorem it is possible to recover the known values for the Box counting dimension of sets containing skeletons centered at any point of a prescribed set of centers under the condition of having full Box dimension. The lower bound relies on a specific construction and to obtain the upper bound there is in fact a result for a bigger (normwise) maximal type operator which is localized on a given cube and it is linear (see \eqref{eq:MvsMtilde} and \cite[Lemma 2.6]{OS18} for the original source).

\subsection{Weighted estimates for averaging operators} The main question on the subject of weighted estimates for a given operator $T$ bounded on $L^p(\mathbb{R}^n)$ for some $p\ge 1$ is to characterize those non negative locally integrable functions $w:\mathbb{R}^n\to \mathbb{R}_{\ge 0}$ such that $T$ maps boundedly $L^p(\mathbb{R}^n,wdx)$ to  $L^p(\mathbb{R}^n,wdx)$. We will denote this spaces by $L^p(w)$ for short. The most famous problem in this direction is the celebrated Muckenhoupt's theorem \cite{Muckenhoupt:Ap} stating that the usual Hardy-Littlewood maximal function (over cubes, for example) defined by 
\begin{equation*}
Mf(x)=\sup_{Q\ni x}\avgint_Q |f(y)| dy
\end{equation*}
 is bounded on  $L^p(w)$  if and only if the weight $w$ satisfies the so called $A_p$ condition given by
\begin{equation}\label{eq:Ap}
[w]_{A_p}:=\sup_Q \left(\avgint_Q w\,dx\right) \left(\avgint_Q w^{1-p'}\,dx\right)^{p-1} <\infty.
\end{equation}
Here, as usual, teh symbol $\avgint_E f d\mu:=\frac{1}{\mu(E)}\int_E fd\mu$ denotes the average of the function $f$ with respect to the measure $\mu$ over the set $E$. This condition also characterizes the boundedness on $L^p(w)$ of other important operators in analysis such as the Hilbert transform, the Riesz transforms and very general Calder\'on-Zygmund operators. More recently, the interest was focused on obtaining precise quantitative estimates for the operator norm in weighted $L^p$ spaces. The first fundamental result on this direction is due to Buckley \cite{Buckley} and is the following norm inequality. For $1<p<\infty$, we have that
\begin{equation}\label{eq:HL-Maximal}
\|Mf\|_{L^p(w)}\le [w]^{\frac{1}{p-1}}_{A_p}\|f\|_{L^p(w)}.
\end{equation}

Other important contribution to precise weighted estimates can be found in \cite{Hytonen:A2,HL-Ap-Ainf,HP} (and the references therein) for the case of cubic $A_p$ weights and several classical operators.

It is then completely natural to ask for classes of weights characterizing the boundedness of the averaging operators discussed in the present paper. As an ideal result, one could expect to find a geometric condition similar to the $A_p$ from \eqref{eq:Ap}. This is in fact possible when dealing with euclidean or metric balls, rectangles or more general bases of convex sets.

However, the geometric nature of the Kakeya or spherical maximal operators is much more difficult to handle and such a clean condition is not known. For the Kakeya maximal operator, some results on weighted estimates can be found in  \cite{Muller1990} where the author obtains a rather involved geometric condition. 

Another related problem on weighted inequalities is the so called Fefferman-Stein type inequality. Given some (maximal) operator $\mathcal{M}$, the problem is to find the best possible range of exponent for which the following inequality is valid for any non negative weight $w$:
\begin{equation*}\label{eq:Fefferman-Stein}
\|\mathcal{M}f\|_{L^p(w)}\lesssim\|f\|_{L^p(\mathcal{M}w)},
\end{equation*}
where the weight on the left hand side is different from the one on the right hand side, being the latter the one obtained by applying the operator itself. The results in  \cite{Vargas-WeightedKakeya, MullerSoria93, Tanaka2001} are in this direction for the case of Kakeya. For  the  spherical maximal function, we refer the reader to \cite{DV96} and \cite{DuoSeijo} for several results concerning power weights of the form $w(x)=|x|^{\alpha}$. Here, as in the case of Kakeya, there is no geometric $A_p$ condition as in the case of HL maximal function.

\subsection{Our contribution}
In this paper we will show that there is, for the skeleton maximal function defined in \eqref{eq:SkeletonMaximal}, a necessary condition that resembles the $A_p$ condition from \eqref{eq:Ap}. In addition, we also find a (different) sufficient condition for the boundedness that also provides a precise quantitative dependence on an $A_p$-like constant. This latter result is obtained following some ideas 
from \cite{OS18} introducing a discretization and linearization of the maximal operator. 

\begin{remark}
We choose to present our results in the setting of 1-skeletons in $\mathbb{R}^2$ for the sake of clarity in the exposition. Many of the conclusions can be trivially extended to higher dimension but we will omit here the discussion of such extensions and only mention the results in Section \ref{subsec:higher}.
\end{remark}

We introduce here an $A_p$ type condition suited to our purposes.

\begin{definition}\label{def:Ap-nonlinear}
Let $w$ be a weight and $1<p<\infty$. 
Let $\mathcal{S} \subseteq \rr^2$ be the family of all 1-skeletons (boundaries of squares) with center in $x \in \rr^2$ and sidelength $1 \leq r \leq 2$. Each element of this family $\mathcal{S}$ will be denoted as  $S(x,r)$. Let $\delta>0$ be a fixed parameter and recall that $S_\delta (x,r)$ denotes a $\delta$-neighborhood of $S(x,r)$  and $S^j_\delta (x,r)$, $j=1,\ldots,4$, each one of the four fattened sides. We say that $w$ belongs to the \emph{skeleton} $\mathcal{A}^S_{p}$ class if 
\begin{equation}\label{eq:Ap-nonlinear}
 \sup_{S(x,r) \in \mathcal{S}} \left(\frac{1}{|S_\delta (x,r)|}\int_{Q_\delta(x)} w \right) \left(\avgint_{\ell_{\delta}(x,r)} w^\frac{-p'}{p} \right)^{p} \left( \avgint_{S_\delta(x,r)} w^\frac{-p'}{p} \right)^{-1} < \infty,
\end{equation}
where $Q_\delta(x)$ is a square of sidelength $\delta$ centered at $x$ and $\ell_{\delta}(x,r)$ denotes the side where the following minimum 
$$\min_{j=1}^4 \frac{1}{|S^j_\delta (x,r)|}\int_{S^j_\delta (x,r)} w^\frac{-p'}{p} \, dx$$
is attained. When finite, the above supremum will be denoted by $[w]_{\mathcal{A}^S_{p}}$.
\end{definition}

Our first theorem is the following, regarding a necessary condition for the boundedness of $M^{1}_{\delta}$ defined in \eqref{eq:SkeletonMaximal}. From now on, we simply denote $M_\delta$ omitting the superscript.

\begin{theorem}\label{thm:necessary}
Let $w$ be a weight in $\mathbb{R}^2$ and suppose that the maximal operator $M_\delta$, for some fixed $\delta$ is bounded on $L^p(w)$ for $1<p<\infty$. Then the weight $w$ belongs to $\mathcal{A}^S_{p}$.
\end{theorem}

Now, regarding the sufficiency, we need to introduce some extra notation to present an auxiliary operator which is a linearized and discretized version of the skeleton maximal operator. We start by recalling this lemma from \cite{OS18} that chooses wisely one $k$-face on each one of the member of a large family of skeletons avoiding large overlaps.

\begin{lemma}\label{lem:chooseelement}
There is a constant $C_{n,k}<\infty$, depending only on $n,k$, such that the following holds. Let $\{S_k(x_i,r_i)\}_{i=1}^u$ be a finite collection of $k$-skeletons in $\rr^n$. Then it is possible to choose one $k$-face of each skeleton with the following property: If $V$ is an affine $k$-plane which is a translate of a coordinate $k$-plane, then $V$ contains at most
\[
C_{n,k}\displaystyle u^{1-\frac{(n-k)(2n-1)}{2n^2}}
\]
of the chosen $k$-faces.
\end{lemma}

We may now introduce a new maximal operator that controls $M_\delta$ at a local level. To that end, we consider the dyadic lattice $\mathbb{Z}^2$ and, given a fixed point $z\in \mathbb{Z}^2$, we shall use the notation from the following definition.
\begin{definition}\label{def:delta-grid}
Let $z$ be any point in the lattice $\mathbb{Z}^2$. Associated to that point we define:
\begin{enumerate}
\item $Q_z$ will denote the square of sidelength 1 and lower left vertex $z$.

\item $Q_z^* = \{x_1,\ldots,x_u \}$ will be the set of centers of the smaller squares of sidelength $\delta$, $Q_z (1), \ldots, Q_z(u)$, generated by the grid $Q_z \cap \delta \Z^2$. We may assume from now on that $1/\delta$ is an integer and therefore $u=\frac{1}{\delta^2}$.
 
 \item $\rho: Q_z^* \rightarrow [1,2]\cap \delta \Z$ is a function that assigns a number as sidelenght for building squares with centers in the points in $Q^*_z$. We denote by $\Gamma$ the family of all possible choices of $\rho$.
 
 \item $ \{S^{\rho}_i\} =\{S(x_i,\rho(x_i))\}_{i=1}^u$ is the family of 1-skeletons with center $x_i$ and sidelength $2\rho(x_i)$.
 
 \item $\Phi$ will be a function given by the selection algorithm given by Lemma \ref{lem:chooseelement} that chooses one of the sides of 
 each $1$-skeleton:  
 \[
 \Phi (S^{\rho}_i) = l_i.
 \]
 \item As the family of sides $l_i$ chosen by the previous algorithm depends on $\rho$ and $Q_z$
we will denote it as
 \[
 \{l^{\rho}_{1,z},\ldots, l^{\rho}_{u,z}\}.
 \]

 \item Note that the main objects we are going to deal with are averages over $\delta$-neighborhoods of these line segments.  We will denote by $\ell^{\rho}_{i,z}$ the $\delta$-neighborhood of  $l^{\rho}_{i,z}$. Since the eccentricity parameter $\delta>0$ is fixed from the beginning we will omit it in the above notation.
 \item Given $x \in Q_z$, we will denote by $x^*$ the center of the square $Q_z(i)$ containing $x$. For notational convenience, we write $\Psi: Q_z \rightarrow Q_z^*$, $x \mapsto x^*$.
 This function maps each one of the squares $Q_z(i)$ to its center $x_i$.
  \end{enumerate}
\end{definition}

\begin{definition} \label{def:linearmaximal}  Let $Q_z$ be a fixed square given by a point in the lattice $\mathbb{Z}^2$. Consider a function $\rho\in\Gamma$ and $0<\delta <1$, if $f \in L^{1}_{\rm loc}(\rr^n)$ we define the  \textit{$(\rho,1)$-skeleton maximal function} with width $\delta$ as

\begin{center}
\begin{equation*}
{\widetilde{M}}^{1}_{\rho,\delta}f : Q_z \rightarrow \R,
\end{equation*}
\begin{equation*}
{\widetilde{M}}^{1}_{\rho,\delta}f(x)= \displaystyle\frac{1}{|\ell_{x,\delta}|} \int _ {\ell_{x,\delta}} f(y)\, dy,
\end{equation*}
\end{center}
where $\ell_{x,\delta}$ is the $\delta$-neighborhood of $\Phi_{\rho} (S_{k}(x^*, \rho(x^*)))$.
\end{definition}

By definition,  $\widetilde{M}^1_{\rho,\delta}$ is a linear operator, and from now on we will simple denote it by ${\widetilde{M}}_{\rho,\delta}$ .
 Let $C Q_z$ denote the square with the same center as $Q_z$ and side length $C$. Since the function $\rho$ only take values in the interval $[1,2]$, in the previous definition it is enough to consider functions $f$ supported on $7Q_z$, because if $x \in Q_z$ then $S_{k,\delta}(x^*,\rho(x^*)) \subseteq 7Q_z$.
The relevance of this operator comes from a pointwise inequality from \cite[Lemma 2.6]{OS18} involving
${\widetilde{M}}_{\rho,\delta}$ and $M_{\delta}$. More precisely, we have that 
\begin{equation}\label{eq:MvsMtilde}
\|M_{\delta}f\|_{L^p(Q_z,w)} \leq 3 \sup_{\rho\in \Gamma}\|\widetilde{M}_{\rho, 3\delta}f\|_{L^p(Q_z,w)}.
\end{equation}
Therefore, by obtaining a sufficient condition on the discrete maximal operator uniformly on $\rho$, we will also obtain sufficiency result for $M_{\delta}$, at least at a local level. 

The appropriate definition of an $A_p$ type class for the discrete maximal operator is the following.

\begin{definition}
Let $\delta>0$ be a fixed eccentricity parameter and let $w$ be a weight in $\mathbb{R}^2$ and let $1<p<\infty$. For a given $z\in \mathbb{Z}^2$ and $\rho \in \Gamma$, we define the skeleton $A_p$-like quantity (we omit the implicit $\delta$ in the following definition):
\begin{equation}\label{eq:Ap-rho-z}
  [w]_{A^S_{p,\rho,z}} := \max_{i=1}^u  \left(\frac{1}{|\ell^{\rho}_{i,z}|}\int_{Q_z(i)} w \right) \left( \avgint_{\ell^{\rho}_{i,z}} w^\frac{-p'}{p}\right)^{p-1},
\end{equation}
where each $\ell^\rho_{i,z}$ is the $\delta$-neighborhood of the edge chosen by the function $\Phi$ from Definition \ref{def:delta-grid}, item (5).
 Taking into account that we need a uniform control over all $z$ and all $\rho$, we define
\begin{equation}\label{eq:ApSkeleton}
[w]_{A^S_p}: = \sup_{z \in \mathbb{Z}^2}  \sup_{\rho} \max_{i=1}^u  \left(\frac{1}{|\ell^{\rho}_{i,z}|}\int_{Q_z(i)} w \right) \left( \avgint_{\ell^{\rho}_{i,z}} w^\frac{-p'}{p}\right)^{p-1}.
\end{equation}
As usual, we say that $w\in A^S_p$ if the above supremum is finite. Note that the geometry of the problem is reflected in the asymmetry on the averages.
\end{definition}

\begin{definition} (The case $p=1$)
 Let $w$ a weight in $\mathbb{R}^2$. For a given $z \in \Z^2$ and $\rho \in \Gamma$ we define de skeleton $A_1$-like quantity: 
 \[
[w]_{A^S_{1,\rho,z}} :=  \max_{i=1}^u \left( \frac{1}{|\ell^{\rho}_{i,z}|} \int_{Q_z(i)} w \right)\norm{w^{-1}}_{L^{\infty}(\ell^{\rho}_{i,z})}.
\]
To have a uniform control over all $z$ and all $\rho$, we define
\[
[w]_{A^S_{1}} := \sup_{z \in \Z^2} \sup_{\rho} \max_{i=1}^u \left( \frac{1}{|\ell^{\rho}_{i,z}|} \int_{Q_z(i)} w \right)\norm{w^{-1}}_{L^{\infty}(\ell^{\rho}_{i,z})}.
\]
\end{definition}

 We will prove in Proposition \ref{prop:ApProperties} that (as in the classical case) the limit when $p\to 1$ of the quantity $[w]_{A^S_p}$ for a weight $w\in A^S_1$ is precisely $[w]_{A^S_1}$.

We may now present the second main result in this article: a sufficient condition for the weighted boundedness of the maximal operator $\widetilde{M}_{\rho,\delta}$ and therefore also for $M_\delta$. The method of proof differs from the usual techniques from the theory of weights since there are no classical covering arguments, instead we use a duality argument that allows to pass from packing objects with good $L^{p'}$ norm to an $L^p$ bound of the maximal operator (see \cite[Proposition 22.4]{Mattila-book2015}). This forces, as far as we know, to only obtain a discrete family of values of $p$'s for the $L^p(w)$ bound (see the discussion in Section \ref{sec:extension-all-p} for further observations on this problem).

We have the following theorem. 

\begin{theorem}\label{thm:sufficient-local}
Let $w$ be a weight in $\mathbb{R}^2$. If $w\in A^S_{p,\rho,z}$ and $p'\in \mathbb{N}$, then 
\begin{equation}
\|\widetilde{M}_{\rho,\delta}f\|_{L^{p}(Q_z,w)} \leq  C \delta^{-\frac{5}{4p}}  [w]^{\frac{1}{p}}_{A^S_{p,\rho,z}} \norm{f}_{L^{p}(7Q_z,w)}.
\end{equation}
\end{theorem}

Note that this bound doesn't blow up when $p$ goes to 1. This, together with the result in Proposition \ref{prop:ApProperties} mentioned above, allows us to use a limiting argument to extend the result to the endpoint, namely to $L^1$.
\begin{corollary} \label{cor: case p=1}  Let $w$ be a weight in $\rr^2$. If $w \in A^S_{1,\rho,z}$, then 
\[
\|\widetilde{M}_{\rho,\delta}f\|_{L^{1}(Q_z,w)} \leq  C \delta^{-\frac{5}{4}}  [w]_{A^S_{1,\rho,z}} \norm{f}_{L^{1}(7Q_z,w)}.
\]
\end{corollary}

As a consequence of this theorem, by using inequality \eqref{eq:MvsMtilde} and the bounded overlap of the family $\{7Q_z\}_{z\in \mathbb{Z}^2}$, we obtain the desired sufficient condition for the skeleton maximal function. We remark here that the proof provides, quite surprisingly, a precise quantitative dependence on the $A_p$ constant. 

\begin{theorem}\label{thm:sufficient} 
Let $w\in A^S_p$ in $\mathbb{R}^2$ for $p$ such that $p'\in \mathbb{N}$ or $p=1$. For any fixed $\delta>0$, we have that
\begin{equation}\label{eq:LpBoundM}
\|M_\delta f\|_{L^p(w)}\le C \delta^{-5/4p}[w]^{1/p}_{A^S_p}\|f\|_{L^p(w)}.
\end{equation}
\end{theorem}

\subsection{Outline}
The paper is organized as follows. In Section \ref{sec:prelim} we introduce some extra notation and a fundamental proposition on packing of thin tubes yielding norm bounds for our maximal operator. In Section \ref{sec:proofs} we present the proofs of our results.

\section{Preliminaries}\label{sec:prelim}

In this section we introduce some extra notation necessary to our proofs. Recall that the algorithm from Lemma \ref{lem:chooseelement} chooses some preferred sides on each square. So each one of the chosen sides is either a vertical line or an horizontal line. We define 
 \[
 E_{\pi_1}:= \{ x_i \in Q^*_z :  l^{\rho}_{i,z} \, \, \text{is vertical}\}
 \]
 \[
 E_{\pi_2}:= \{ x_i \in Q^*_z :  l^{\rho}_{i,z} \, \, \text{ is horizontal}\}
 \]
 By Definition \ref{def:delta-grid} we have that $\Psi^{-1}(E_{\pi_j})$, $j=1,2$ is a partition of $Q_z$. 

\medskip

The next proposition contains a key result that allows to reduce the problem of finding norm bounds for a maximal operator to the problem of proving an estimate for the sum of indicator functions. A version of this proposition for the Kakeya maximal operator can be found in \cite{Mattila-book2015} and an adaptation of this argument to the unweighted skeleton maximal operator is presented in \cite{OS18}.

 \begin{proposition}\label{pro:normaindicadoras}
Let $0<\delta <1$, $z \in \Z^2$,  $\rho\in \Gamma$ and consider the squares $Q_z(i)$, $1\le i\le u$ from Definition \ref{def:delta-grid}. Fix $1<p<\infty$ and a weight $w$ on $Q_z$. Now consider  only the sides $\ell^{\rho}_{i,z}= \ell_i$ such that $x_i\in  E_{\pi_1}$ ($z,\rho$ are fixed) and denote $|E_{\pi_1}|=v$. Let 
$\Psi^{-1}(E_{\pi_1})=\bigcup_{j=1}^v Q_z(i_j)$.
Suppose that there exists some constant $0<K<\infty$ such that 
\begin{equation*}
\norm{\sum_{i=1}^v t_i\1_{\ell_i}}_{L^{p'} (7 Q_z , w^{1-p'})} \leq K.
\end{equation*}

Then we have
\[
 \norm{\widetilde{M}_{\rho,\delta}f}_{L^p(\Psi^{-1}(E_{\pi_1}), w)} \leq K  \norm{f}_{L^p(7 Q_z, w)}.
\]
\end{proposition}

\begin{proof}
For the sake of clarity, we relabel the cubes in $\Psi^{-1}(E_{\pi_1})$ as $Q_z(i)$, $i=1,\ldots, v$ and compute the integral
\begin{eqnarray*}
\int_{\Psi^{-1}(E_{\pi_1})} \widetilde{M}_{\rho,\delta} f(x)^p w(x) \, dx &=& \displaystyle\sum_{i=1}^v \int_{Q_z(i)} \widetilde{M}_{\rho,\delta} f(x)^p w(x)\, dx\\& =& \sum_{i=1}^v \widetilde{M}_{\rho,\delta}f(x_i)^p w(Q_z(i)).   
\end{eqnarray*}
Let us wirte $\ell_i$ instead of $\ell^{\rho}_{i,z}$. By duality there is a choice of nonegative numbers $b_i$, $i=1,\dots, v$ with $\displaystyle\sum_{i=1}^v b^{p'}_i=1$, such that
\begin{eqnarray*}
\left (\int_{\Psi^{-1}(E_{\pi_1})} \widetilde{M}_{\rho,\delta}f(x)^p w(x) \, dx \right)^{1/p} &= & \left( \sum_{i=1}^v (\widetilde{M}_{\rho,\delta}f(x_i) w(Q_z(i))^{1/p})^p  \right)^{1/p} \\ &=&
\sum_i^v \widetilde{M}_{\rho,\delta}f(x_i) w(Q_z(i))^{1/p} b_i \\&=& \sum_{i=1}^v \frac{b_i}{|\ell_i|}  \left( \int_{Q_z(i)} w \, dx  \right)^{1/p}\left( \int_{\ell_i} f(y) \, dy\right) \\&=& \sum_{i=1}^v t_i \left( \int_{\ell_i} f(y)\, dy\right)
\end{eqnarray*}
where we write  $t_i = \frac{b_i}{|\ell_i|}  \left( \int_{Q_z(i)} w \, dx  \right)^{1/p} $.

Then, we obtain that
\begin{eqnarray*}
\norm{\widetilde{M}_{\rho,\delta}f}_{L^p(\Psi^{-1}(E_{\pi_1}))} &\leq& \displaystyle\sum_{i=1}^v t_i \left( \int_{\ell_i} f(y) dy\right) \\ &\leq& \int_{7Q_z} \left( \sum_{i=1}^v t_i \1_{\ell_i}\right) f w^{1/p} w^{-1/p} \,dx \\ &\leq& \left(\int_{7Q_z} f^p w \, dx \right)^{1/p}\left(\int_{7Q_z} \left(\sum_{i=1}^v t_i \1_{\ell_i}\right)^{p'} w^{-p'/p} \right)^{1/p'} \\ &\leq& \norm{f}_{L^p(7Q_z, w)} \norm{\sum_{i=1}^v t_i \1_{\ell_i}}_{L^{p' }(7Q_z, w^{-p'/p})}\\
&\leq& K \norm{f}_{L^p(7Q_z, w)}
\end{eqnarray*}
since, by hypothesis, the last norm of the sum of indicator functions is bounded. The proof is complete.

\end{proof}

We include here some properties of the weights.

\begin{proposition}\label{prop:ApProperties} 
Let $w \in A^{S}_{p}$ for some $1 \leq p <\infty$. Then
\begin{enumerate}
    \item The classes $A^{S}_{p}$ are increasing as $p$ increases; precisely, for $1\leq p < q < \infty$ we have
    \[
     \left[ w \right]_{A^{S}_{q}} \leq \left[ w \right]_{A^{S}_{p}}.
    \]
    \item $\displaystyle \lim_{p \rightarrow 1^{+}}\left[w \right]_{A^{S}_{p}}=  \left[ w \right]_{A^{S}_{1}}$.
\end{enumerate}
\end{proposition}

\begin{proof}
 Let $1<p<q$, we shall prove that $A^{S}_{p} \subseteq A^{S}_{q}$. First, observe that: if $p<q$, then $0< q'-1<p'-1<\infty$ and if $1< r=\frac{p'-1}{q'-1}$, then $r'=\frac{p'-1}{p'-q'}$. Then 
    \begin{eqnarray*}
    \left(\frac{1}{|\ell^{\rho}_{i,z}|}\int_{\ell^{\rho}_{i,z}} w^{1-q'}\right)^{q-1} 
    &=& \frac{1}{|\ell^{\rho}_{i,z}|^{q-1}} \left(  \int_{\ell^{\rho}_{i,z}} (w^{-1})^{q'-1}\1_{\ell^{\rho}_{i,z}}\right)^{\frac{1}{q'-1}} 
    \\ &\leq& \frac{1}{|\ell^{\rho}_{i,z}|^{q-1}}\left(\left( \int_{\ell^{\rho}_{i,z}} (w^{-1})^{r(q'-1)}\right)^{1/r} \left( \int_{\ell^{\rho}_{i,z}} \1^{r'}_{\ell^{\rho}_{i,z}}\right)^{1/r'}\right)^{\frac{1}{q'-1}} 
    \\ &\leq& \frac{1}{|\ell^{\rho}_{i,z}|^{q-1}}\left( \int_{\ell^{\rho}_{i,z}} (w^{-1})^{p'-1}\right)^{\frac{1}{p'-1}} |\ell^{\rho}_{i,z}|^{\frac{p'-q'}{p'-1}\frac{1}{q'-1}} 
    \\ &\leq&
 \frac{1}{|\ell^{\rho}_{i,z}|^{q-1}}\left( \int_{\ell^{\rho}_{i,z}} (w^{-1})^{p'-1}\right)^{\frac{1}{p'-1}} |\ell^{\rho}_{i,z}|^{q-p} 
 \\ &\leq&  \frac{1}{|\ell^{\rho}_{i,z}|^{p-1}}\left( \int_{\ell^{\rho}_{i,z}} w^{1-p'}\right)^{p-1},
    \end{eqnarray*}
    for every $i = 1,\ldots, u$, $\rho \in \Gamma$ and $z \in \Z^2$.
 Since $w \in A^{S}_{p}$, 
we have that
\begin{eqnarray*}
   \left[ w \right]_{A^{S}_q} & = & \sup_{z \in \mathbb{Z}^2}  \sup_{\rho} \max_{i=1}^u  \left( \frac{1}{|\ell^{\rho}_{i,z}|}\int_{Q_z(i)} w dx \right) \left(\frac{1}{|\ell^{\rho}_{i,z}|}\int_{\ell^{\rho}_{i,z}} w^{1-q'}\right)^{q-1} \\
   & \le &\sup_{z \in \mathbb{Z}^2}  \sup_{\rho} \max_{i=1}^u 
    \left( \frac{1}{|\ell^{\rho}_{i,z}|}\int_{Q_z(i)} w dx \right) \left(\frac{1}{|\ell^{\rho}_{i,z}|}\int_{\ell^{\rho}_{i,z}} w^{1-p'}\right)^{p-1} \\
    & = &  \left[ w \right]_{A^{S}_p}  <\infty.
\end{eqnarray*}

For the case $p=1$, we use that
\[
 \left(\frac{1}{|\ell^{\rho}_{i,z}|}\int_{\ell^{\rho}_{i,z}} w^{1-q'}\right)^{q-1}  \leq \sup_{x \in \ell^{\rho}_{i,z}} w(x)^{-1} =
\norm{w^{-1}}_{L^\infty (\ell^{\rho}_{i,z})}.
\]

To prove the second item, using 1, we have $[w]_{A^{S}_p} \leq [w]_{A^{S}_1}$ for all $p>1$.
It is enough to prove that given $\varepsilon>0$ there exists $\delta>0$ such that 
\[
\left[ w \right]_{A^{S}_p} \geq \left[ w \right]_{A^{S}_1} -\varepsilon, 
\]
for any $p<1+\delta$.
Observe that
\begin{align*}
\left(\frac{1}{|\ell^{\rho}_{i,z}|}\int_{\ell^{\rho}_{i,z}} w^{1-p'} dx  \right)^{p-1} 
&= \left(\frac{1}{|\ell^{\rho}_{i,z}|}\int_{\ell^{\rho}_{i,z}} (w^{-1})^{p'-1} dx  \right)^{\frac{1}{p'-1}} 
\\ &= \frac{1}{|\ell^{\rho}_{i,z}|^{p-1}} \norm{w^{-1}}_{L^{p'-1}(\ell^{\rho}_{i,z})}
\\ &=  \frac{1}{|\ell^{\rho}_{i,z}|^{\frac{1}{r}}} \norm{w^{-1}}_{L^{r}(\ell^{\rho}_{i,z})},
\end{align*}
where in the last line we write $r=p'-1$.
It is known that $\norm{f}_{L^r} \rightarrow \norm{f}_{L^\infty}$ if $r \rightarrow \infty$.
Then, given $\varepsilon>0$, there exists $N \in \mathbb{N}$ such that if $r > N$, 
\begin{eqnarray*}
\norm{w^{-1}}_{L^\infty(\ell^{\rho}_{i,z})} \leq \frac{1}{|\ell^{\rho}_{i,z}|^{1/r}} \norm{w^{-1}}_{L^r(\ell^{\rho}_{i,z})} + \tfrac{\varepsilon}{2}.
\end{eqnarray*}
Or, we can say there exists $\delta>0$ such that 
\[
\norm{w^{-1}}_{L^\infty(\ell^{\rho}_{i,z})} \leq \left(\frac{1}{|\ell^{\rho}_{i,z}|}\int_{\ell^{\rho}_{i,z}} w^{1-p'} dx  \right)^{p-1} + \tfrac{\varepsilon}{2},
\]
if $p<\delta +1$. Let $\varepsilon>0$ as above, then there exists $\rho \in \Gamma$, $z \in \Z^2$ and $i$ such that  
 \begin{align*}
 \left[w\right]_{A^{S}_1} -\varepsilon/2 
 & \leq \left(\frac{1}{|\ell^{\rho}_{i,z}|}\int_{Q_z(i)}w dx \right) \norm{w^{-1}}_{L^\infty(\ell^{\rho}_{i,z})}
 \\ &\leq \left(\frac{1}{|\ell^{\rho}_{i,z}|}\int_{Q_z(i)}w dx \right) \left(\frac{1}{|\ell^{\rho}_{i,z}|}\int_{\ell^{\rho}_{i,z}} w^{1-p'} dx  \right)^{p-1} + \tfrac{\varepsilon}{2}.  
 \end{align*}
Since $w \in A^S_p$, taking supremum we obtain 
\[
 \left[w\right]_{A^{S}_1}- \varepsilon  \leq  \left[w\right]_{A^{S}_p}. 
\]
\end{proof}

\section{Proofs of the main results}\label{sec:proofs}

We start this section wit the proof of Theorem \ref{thm:necessary}, that is, a necessary condition on the weight $w$ for the maximal operator $M_\delta$ to be bounded on $L^p(w)$.

\begin{proof}[Proof of Theorem \ref{thm:necessary}]
Let us denote by $\mathcal{S} \subseteq \rr^2$ the family of all boundaries of squares centered at any $x \in \rr^2$ and sidelength $2r$, $1 \leq r \leq 2$. Each element in $\mathcal{S}$ will be denoted by $S(x,r)$. Now, given $S(\tilde{x},r) \in \mathcal{S}$, lets consider $f\1_{S_\delta (\tilde{x},r)}$.

Let $Q_{\tilde{x}}$ be the square centered at $\tilde{x}$ with sidelength $\delta$. Then, for any $x$ in $Q_{\tilde{x}}$ we have that
\[
S_\delta (\tilde{x}, r) \subseteq S_{4\delta}(x,r).
\]
Then
 \begin{eqnarray*}
\min_{j=1}^4 \avgint_{S^j_\delta (\tilde{x},r)} |f\1_{S_\delta (\tilde{x},r)} (y)| \, dy &\lesssim & \min_{j=1}^4 \avgint_{S^j_{4\delta} (x,r)} |f\1_{S_\delta (\tilde{x},r)}(y)| \, dy \\ &\leq& M_{4\delta}\left( f\1_{S_\delta (\tilde{x},r)}\right )(x).
\end{eqnarray*}

Using this estimate and the hypothesis on $M_\delta$, we have that 
 \begin{eqnarray*}
w(Q_{\tilde{x}}) \left( \min_{j=1}^4 \avgint_{S^j_\delta (\tilde{x},r)} |f\1_{S_\delta (\tilde{x},r)}(y)| \, dy\right)^p &\leq & \int_{Q_{\tilde{x}}} |M_{4\delta}f\1_{S_\delta (\tilde{x},r)}(x)|^p w(x) \, dx \\ &\leq&
 C(4\delta,w) \int_{S_\delta (\tilde{x},r)} |f(x)|^p w(x) \, dx  \end{eqnarray*}
and therefore
\[
w(Q_{\tilde{x}}) \left( \min_{j=1}^4 \avgint_{S^j_\delta (\tilde{x},r)} |f(y)| \, dy\right)^p \leq
C(4\delta,w) \int_{S_\delta (\tilde{x},r)} |f(x)|^p w(x) \, dx,
\]
for any $f$ and $r \in [1,2]$.

\medskip

Let us choose $f= w^{-p'/p}$ in the previous inequality. Then 

\[
w(Q_{\tilde{x}}) \left( \min_{j=1}^4 \avgint_{S^j_\delta (\tilde{x},r)} w^{-p'/p} \, dx\right)^p \leq
C(4\delta,w) \int_{S_\delta (\tilde{x},r)} w^{-p'/p} \, dx,
\]

Let us denote by $\ell_{\delta}(\tilde{x},r)$ the side $j$ for which the minimum above is attained. Then, under the hypothesis of the weight being non trivial on $S_\delta (\tilde{x},r)$ (if not, a standard argument as in the classical case can be used to overcome this obstacle), we can write 
\[
w(Q_{\tilde{x}}) \left(\avgint_{\ell_{\delta}(\tilde{x},r)} w^{-p'/p} \, dy\right)^p \left( \int_{S_\delta(\tilde{x},r)} w^{-p'/p} \, dx.\right)^{-1} \leq
 C (4\delta,w),
\]
which is equivalent to
 \[
 \frac{w(Q_{\tilde{x}})}{|S_\delta (\tilde{x},r)|} \left(\avgint_{\ell_{\delta}(\tilde{x},r)} w^{-p'/p} \, dx\right)^p \left( \avgint_{S_\delta(\tilde{x},r)} w^{-p'/p} \, dx.\right)^{-1} \leq
 C (\delta,w).
 \]
Taking the supremum over $\mathcal{S}$, we obtain
 \begin{align*}
\sup_{S(\tilde{x},r) \in \mathcal{S}} \left(\frac{1}{|S_\delta(\tilde{x},r)|}\int_{Q_{\tilde{x}}} w \, dx\right) \left(\avgint_{\ell_{\delta}(\tilde{x},r)} w^{-\frac{p'}{p}} \, dx\right)^{p} \left( \avgint_{S_\delta(\tilde{x},r)} w^{-\frac{p'}{p}}  dx\right)^{-1} &\leq C(\delta,w).
\end{align*}
And this is precisely the $\mathcal{A}^S_p$ condition \eqref{eq:Ap-nonlinear} from Definition \ref{def:Ap-nonlinear}.
\end{proof}

Now we present the proofs of Theorem \ref{thm:sufficient-local} and Theorem \ref{thm:sufficient}, namely the sufficiency of the $A^S_p$ condition for the boundedness of the operators $\widetilde{M}_{\rho,\delta}$ and $M_\delta$.

\begin{proof}[Proof of Theorem \ref{thm:sufficient-local}]
Consider $Q_z = \Psi^{-1}(E_{\pi_1}) \cup \Psi^{-1}(E_{\pi_2}) $. We will consider only $E_{\pi_1}$, the other set can be treated in the same way. By proposition \ref{pro:normaindicadoras}, it is sufficient to prove that
\[
 \norm{\sum_{i=1}^v t_i \1_{\ell_i}}_{L^{p'} (7Q_z , w^{1-p'})} \leq C \delta^{-\frac{5}{4p}}  [w]_{A^S_{p,\rho,z}}^{1/p},
\]
where $v$ is the cardinality of $|E_{\pi_1}|$ and the $t_i$'s are defined by 
\begin{equation}\label{eq:tis2}
t_i = \frac{b_i}{|\ell_i|}  \left( \int_{Q_z(i)} w \, dx  \right)^{1/p}, \qquad \sum_{i=1}^v b_i^{p'}=1.
\end{equation}

The condition will be satisfied if we can prove the following equivalent formulation:
\begin{equation*}
\norm{\sum_{i=1}^v t_i\1_{\ell_i}}^{p'}_{L^{p'} (7Q_z w^{1-p'})} 
\leq \delta^{-\frac{5p'}{4p}}  [w]^{p'/p}_{A^S_{p,\rho,z}}.
\end{equation*}
Recall that $p'$ is a positive integer, so we have that
\begin{eqnarray*}
\norm{\sum_{i=1}^v t_i\1_{\ell_i}}^{p'}_{L^{p'}(7Q_z, w^{1-p'})} &=& \sum^{v}_{i_1,\ldots,i_{p'} = 1} t_{i_1}\ldots t_{i_{p'}} w^{1-p'}(\ell_{i_1}\cap\ldots \cap \ell_{i_{p'}})\\ 
&=& \sum^{v}_{i_1,\ldots,i_{p'} = 1} \prod_{j=1}^{p'} t_{i_j} \left(w^{1-p'}(\ell_{i_1}\cap\ldots \cap \ell_{i_{p'}})\right)^{1/p'}\\ 
&\leq& \prod_{j=1}^{p'}\left( \sum^{v}_{i_1,\ldots,i_{p'} = 1} t^{p'}_{i_j}w^{1-p'}(\ell_{i_1}\cap\ldots \cap \ell_{i_{p'}})\right)^{1/p'}
\end{eqnarray*}
In the last inequality we used the discrete  H\"older's inequality. Each of the $p'$ factors obtained have the same bound, so we only look one of them. For example, take the first:
\begin{eqnarray*}
\sum_{i_1,\ldots,i_{p'}=1}^v t_{i_1}^{p'} w^{1-p'}(\ell_{i_1}\cap\ldots \cap \ell_{i_{p'}})&=& \sum_{i_1=1}^v t_{i_1}^{p'} \sum_{i_2,\ldots,i_{p'}=1}^v  w^{1-p'}(\ell_{i_1}\cap\ldots \cap \ell_{i_{p'}}).
\end{eqnarray*}
\medskip

Now, in the last sum, for each fixed  $i_1$, we only have to consider those terms where $\ell_{i_1}\cap\ldots \cap \ell_{i_{p'}}\neq \emptyset$. We have a way to estimate the quantity of neighbors that each $\ell_{i_1}$ has from Lemma \ref{lem:chooseelement} . So, setting $C_{p,\delta}=C^{p'-1} \delta^{-\frac{5}{4}(p'-1)}$, where $C$ is the constant in the lemma, we can write
\[
\sum_{i_2,\ldots,i_{p'}=1}^v  w^{1-p'}(\ell_{i_1}\cap\ldots \cap \ell_{i_{p'}}) 
\leq C_{p,\delta}\, w^{1-p'}(\ell_{i_1}).
\]

Therefore,

\begin{eqnarray*}
 \sum_{i_1=1}^v t_{i_1}^{p'} \sum_{i_2,\ldots,i_{p'}}  w^{1-p'}\left (\bigcap_{j=1}^{p'}\ell_{i_{j}}\right ) &\le & C_{p,\delta}\, \sum_{i_1=1}^v t_{i_1}^{p'}  w^{1-p'}(\ell_{i_1})\\
&\leq& C_{p,\delta}\,  \sum_{i_1=1}^v  w^{1-p'}(\ell_{i_1}) \left(\frac{w(Q_z(i_1))}{|\ell_{i_1}|}\right)^{p'-1}  \frac{b^{p'}_{i_1}}{|\ell_{i_1}|} \\ 
&\leq& C_{p,\delta}\, \left[ w \right]^{\frac{1}{p-1}}_{A^S_{p,\rho,z}}\sum_{i_1=1}^v b_{i_1}^{p'}.
 \end{eqnarray*}

 Finally, we obtain that
\begin{equation*}
\sum_{i_1,\ldots,i_{p'}=1}^v t_{i_1}^{p'} w^{1-p'}(\ell_{i_1}\cap\ldots \cap \ell_{i_{p'}})\le C_{p,\delta}\,\left[ w \right]^{\frac{1}{p-1}}_{A^S_{p,\rho,z}}.
\end{equation*}

Since the same estimate holds for all the other factors involving $t_{i_2},\dots,t_{i_m}$, we conclude that
 \begin{eqnarray*}
  \norm{\sum_i t_i\1_{\ell_i}}_{L^{p'}(w^{1-p'})} &\leq& C^{1/p}\delta^{-\frac{5}{4}\frac{p'-1}{p'}} \left[ w \right]^{\frac{1}{(p-1)p'}}_{A^S_{p,\rho,z}} \\ 
  &\leq& C^{1/p} \delta^{-\frac{5}{4p}} \left[ w \right]_{A^S_{p,\rho,z}}^{1/p}
 \end{eqnarray*}
 and we obtain the desired result.
\end{proof}

Now we are ready to conclude with the proof of Theorem \ref{thm:sufficient} by an easy argument of decomposing $\mathbb{R}^2$ into cubes.

\begin{proof}[Proof of Theorem \ref{thm:sufficient}]
Let us start with the case $p>1$. We compute the $L^p(w)$ norm as
\begin{equation*}
\int_{\rr^2} |{M}_{\delta}f(x)|^p w(x)\, dx = \sum_{z\in\Z^2} \int {\1}_{Q_z} |{M}_{\delta}f(x)|^p w(x)\,dx 
\leq \sum_{z\in\Z^2} \norm{M_\delta f}^p_{L^p(Q_z w)} .
\end{equation*}

Now, using \eqref{eq:MvsMtilde} we may control the maximal operator at a local level by the discretized version as follows
\begin{eqnarray*}
\sum_{z\in\Z^2} \norm{M_\delta f}^p_{L^p(Q_z w)}  &\le& C
\sum_{z\in\Z^2}  \sup_{\rho\in \Gamma}\|\widetilde{M}_{\rho, 3\delta}f\|^p_{L^p(Q_z,w)}\\
&\le& C \delta^{-\frac{5}{4}} \left[ w \right]_{A^S_{p}}\sum_{z\in\Z^2} \norm{f}^p_{L^p(7 Q_z, w)}\\
&\le &  C \delta^{-\frac{5}{4}} \left[ w \right]_{A^S_{p}}\sum_{z\in\Z^2} \norm{f}^p_{L^p( Q_z, w)}
\end{eqnarray*}
since $\norm{\sum_{z\in\Z^n}{\1_{7Q_z}}}_\infty < \infty$. We conclude that
\[ 
\norm{M_\delta f}_{L^p (w)} \lesssim C \delta^{-5/4p}[w]^{1/p}_{A_p} \norm{f}_{L^p(w)} 
\]
and then the proof for $p>1$ is complete. To include the case $p=1$, we simply note that the bound above remains bounded when $p\to1$ by Proposition \ref{prop:ApProperties}.
\end{proof}

\subsection{Some concluding remarks}\label{sec:extension-all-p}
Here we include some final observations and examples to illustrate our results.
\subsubsection{Interpolation}
The reader should note that, unlike the case of cubic Muckenhoupt weights, our result in Theorem \ref{thm:sufficient} was obtained only for those values of $p$ such that $p'$ is an integer. The reason behind that is the key combinatorial Proposition \ref{pro:normaindicadoras}. More precisely, the use of that proposition in the proof of the main Theorem \ref{thm:sufficient} can be carried out expanding the $L^{p'}$ norm and that is possible for those selected values of $p'$. In the unweighted case, the remaining values of $p$ are obtained by interpolation. Here, one could try to interpolate with change of measure. This approach is indeed valid for example in the classical case of cubic $A_p$ weights as it is explained in \cite{Jones80}. But, since it is a consequence of the factorization property of $A_p$ weights into $A_1$ weights, here we do not know if such a factorization is valid or not. The best result that can be obtained by interpolation is that, since  $M$ is bounded in $L^{p_1}(w)$ for all $w\in A^S_{p_1}$ and also in $L^{p_2}(w)$ for all $w\in A^S_{p_2}$, then $M$ is bounded in $L^p(u)$ with
\begin{equation*}
u=w_1^{\frac{\theta}{p_1}}w_2^{{\frac{1-\theta}{p_2}}}\qquad\qquad \frac{1}{p}=\frac{\theta}{p_1}+\frac{1-\theta}{p_2}.
\end{equation*} 
The question here is how to prove that any $A^S_p$ weight can be written in that way \emph{without} using factorization.

\subsubsection{Examples} The following computation shows that  classical cubic $A_p$ weights belong to the skeleton class $\mathcal{A}^S_p$. We have to bound the expression
\[
A=\left(\frac{1}{|S_\delta (\tilde{x},r)|}\int_{Q_{\tilde{x}}}  w \ dx\right) \left(\avgint_{\ell_{\delta}(\tilde{x},r)} w^{-\frac{p'}{p}} dx\right)^{p} \left( \avgint_{S_\delta(\tilde{x},r)} w^{-\frac{p'}{p}} dx\right)^{-1}
\]
uniformly in $S(\tilde{x},r) \in \mathcal{S}$. Let us note that, since $w\ge0$ a.e.,
\begin{eqnarray*}
 \left(\avgint_{\ell_{\delta}(\tilde{x},r)} w^{-\frac{p'}{p}} \, dx\right)^{p} &\leq& \left( \frac{|S_\delta(\tilde{x},r)|}{|\ell_{\delta}(\tilde{x},r)|} \avgint_{S_\delta(\tilde{x},r)} w^{-\frac{p'}{p}} \, dx \right)^p \\ &\leq&  4^p  \left( \avgint_{S_\delta(\tilde{x},r)} w^{-\frac{p'}{p}} \, dx \right)^p.
\end{eqnarray*}
We now use the fact that both $Q_{\tilde{x}}$ and $S_\delta(\tilde{x},r)$ are contained in the cube $C(\tilde{x},2r)$ with center $\tilde{x}$ and sidelenght $4r$. Then,
\begin{eqnarray*}
A & \le & 4^p\left(\frac{1}{|S_\delta (\tilde{x},r)|}\int_{Q_{\tilde{x}}}  w  dx\right)  \left( \avgint_{S_\delta(\tilde{x},r)} w^{-\frac{p'}{p}}dx\right)^{p-1}\\
& \le &4^p\left(\frac{|C(\tilde{x},2r)|}{|S_\delta (\tilde{x},r)|}\avgint_{C(\tilde{x},2r)} w \, dx \right)  \left( \frac{|C(\tilde{x},2r)|}{|S_\delta (\tilde{x},r)|}\avgint_{C(\tilde{x},2r)} w^{-\frac{p'}{p}} \, dx\right)^{p-1} \\
& \le & 4^p \left(\frac{|C(\tilde{x},2r)|}{|S_\delta (\tilde{x},r)|}\right)^p \left(\avgint_{C(\tilde{x},2r)} w \, dx \right)  \left( \avgint_{C(\tilde{x},2r)} w^{-\frac{p'}{p}} \, dx\right)^{p-1}\\
&\le &  4^p \left(\frac{4r^2}{4r\delta}\right)^p \left(\avgint_{C(\tilde{x},2r)} w \, dx \right)  \left( \avgint_{C(\tilde{x},2r)} w^{-\frac{p'}{p}} \, dx\right)^{p-1}   \\
& \le & \frac{4^{2p}}{\delta^p}  \left(\avgint_{C(\tilde{x},2r)} w \, dx \right)  \left( \avgint_{C(\tilde{x},2r)} w^{-\frac{p'}{p}} \, dx\right)^{p-1}.
\end{eqnarray*}
The last expression is precisely controlled by the $A_p$ condition for cubes, $A_p \subset \mathcal{A}^S_{p}$. 

\subsubsection{Higher dimensions}\label{subsec:higher} Every argument presented in this article can be extended to deal with general $k$-skeletons of $n$-dimensional cubes with axes parallel sides in $\mathbb{R}^n$. The proofs are technically involved but straightforward. We include in this last section the results that can be obtained following the same line of ideas. 

Fix $ 0 \leq k < n \in \mathbb{N}$.
Let $\mathcal{S}_k \subset \rr^n$ be the family of all $k$-skeletons with center in $x \in \rr^n$ and sidelength $2r$, $1\leq r \leq 2$.
Each element of this family  will be denoted as  $S_k(x,r)$. Let $\delta>0$ be a fixed parameter and recall that $S_{k,\delta} (x,r)$ denotes the $\delta$-fattened $k$-skeleton and $S^j_\delta (x,r)$ each one of the $k$-faces of the fattened skeleton.

 We can define the $\mathcal{A}^{S_k}_{p}$ class as in Definition \ref{def:Ap-nonlinear}, just considering this new family $\mathcal{S}_k$. More precisely, we say that  $w \in \rr^n$ belongs to $\mathcal{A}^{S_k}_{p}$ class if
\begin{equation*}\label{eq:Ap-nonlinear-k}
 \sup_{S_k(x,r) \in \mathcal{S}_k} \left(\frac{1}{|S_{k,\delta} (x,r)|}\int_{Q_\delta(x)} w \right) \left(\avgint_{\ell_{k,\delta}(x,r)} w^\frac{-p'}{p} \right)^{p} \left( \avgint_{S_{k,\delta}(x,r)} w^\frac{-p'}{p} \right)^{-1} < \infty,
\end{equation*}
where $Q_\delta(x)$ is a $n$-dimensional cube of sidelength $\delta$ centered at $x$ and $\ell_{k,\delta}(x,r)$ denotes the $k$-face where the following minimum 
$$\min_{j=1}^N \frac{1}{|S^j_{k,\delta} (x,r)|}\int_{S^j_{k,\delta} (x,r)} w^\frac{-p'}{p} \, dx$$
is attained. When finite, the above supremum will be denoted by $[w]_{\mathcal{A}^{S_k}_{p}}$.
The following result is the analogous of Theorem \ref{thm:necessary} in the $n,k$ case and the proof is similar. 
\begin{theorem}
Let $w$ be a weight in $\mathbb{R}^n$ and suppose that the maximal operator $M^k_\delta$, for some fixed $\delta$, is bounded on $L^p(w)$ for $1<p<\infty$. Then the weight $w$ belongs to $\mathcal{A}^{S_k}_{p}$.
\end{theorem}

To obtain a sufficient condition, as in the $2$-dimensional case with $k=1$ we introduce a linearized and discretized version of the $k$-skeleton maximal operator $M^k_\delta$ given in \eqref{eq:SkeletonMaximal}.
Note that Definition \ref{def:delta-grid} holds if we replace $\Z^2$ by $\Z^n$ and the $1$-skeletons by $k$-skeletons. 

\begin{definition}   Let $Q_z$ be a fixed square given by a point in the lattice $\mathbb{Z}^n$. Consider a function $\rho\in\Gamma$ and $0<\delta <1$, if $f \in L^{1}_{\rm loc}(\rr^n)$ we define the  \textit{$(\rho,k)$-skeleton maximal function} with width $\delta$ as
\begin{center}
\begin{equation*}
{\widetilde{M}}^{k}_{\rho,\delta}f : Q_z \rightarrow \R,
\end{equation*}
\begin{equation*}
{\widetilde{M}}^{k}_{\rho,\delta}f(x)= \displaystyle\frac{1}{|\ell_{x,\delta}|} \int _ {\ell_{x,\delta}} f(y)\, dy,
\end{equation*}
\end{center}
where $\ell_{x,\delta}$ is the $\delta$-neighborhood of $\Phi_{\rho} (S_{k}(x^*, \rho(x^*)))$.
\end{definition}

The appropriate definition of an $A_p$ type class for the discrete maximal operator in the $n,k$ case is the following.

\begin{definition}
Let $\delta>0$ be a fixed eccentricity parameter and let $w$ be a weight in $\mathbb{R}^n$ and let $1<p<\infty$.
 For a given $z\in \mathbb{Z}^n$ and $\rho \in \Gamma$, we define the $k$-skeleton $A_p$-like quantity (we omit the implicit $\delta$ in the following definition):
\begin{equation}\label{eq:Ap-rho-z-Rn}
  [w]_{A^{S_k}_{p,\rho,z}} := \max_{i=1}^u  \left(\frac{1}{|\ell^{\rho}_{i,z}|}\int_{Q_z(i)} w \right) \left( \avgint_{\ell^{\rho}_{i,z}} w^\frac{-p'}{p}\right)^{p-1}.
\end{equation}

Since we need and uniform control over all $z$ and all $\rho$, we define
\begin{equation}\label{eq:ApSkeleton-Rn}
[w]_{A^{S_k}_p}: = \sup_{z \in \Z^n}  \sup_{\rho} \max_{i=1}^u  \left(\frac{1}{|\ell^{\rho}_{i,z}|}\int_{Q_z(i)} w \right) \left( \avgint_{\ell^{\rho}_{i,z}} w^\frac{-p'}{p}\right)^{p-1},
\end{equation}
We say that $w \in A^{S_k}_p$ is the above supremum is finite. 
\end{definition}
For the case $p=1$ we have the following definition. 
\begin{definition}
 Let $w$ a weight in $\mathbb{R}^n$. For a given $z \in \Z^n$ and $\rho \in \Gamma$ we define de skeleton $A_1$-like quantity: 
 \[
[w]_{A^S_{1,\rho,z}} :=  \max_{i=1}^u \left( \frac{1}{|\ell^{\rho}_{i,z}|} \int_{Q_z(i)} w \right)\norm{w^{-1}}_{L^{\infty}(\ell^{\rho}_{i,z})}.
\]
To have a uniform control over all $z$ and all $\rho$, we define
\[
[w]_{A^{S_k}_{1}} := \sup_{z \in \Z^n} \sup_{\rho} \max_{i=1}^u \left( \frac{1}{|\ell^{\rho}_{i,z}|} \int_{Q_z(i)} w \right)\norm{w^{-1}}_{L^{\infty}(\ell^{\rho}_{i,z})}.
\]
\end{definition}

In higher dimensions we have more than two orientations, not only vertical or horizontal.  The $n$ canonical vectors $e_1,\ldots, e_n$ in $\rr^n$ determine 
$\binom{n}{k}$ coordinate $k$-planes. We will denote these $k$-planes as 
$\pi_1,\dots,\pi_{\binom{n}{k} }.$

Each $k$-face of a $k$-skeleton of an $n$-dimensional cube is contained in an affine $k$-plane which is a translate of some $\pi_{j}$, $1\leq j \leq \binom{n}{k}  $; in this case we say that this $k$-face is parallel to $\pi_{j}$ (in the case $k=0$, the origin is the $0$-plane determined by the axes).

\begin{definition} \label{def:E_pi sets}
Let $\rho\in\Gamma$ and let  $Q^*_z = \{ x_1, \ldots, x_u\}$, for a given $z \in \Z^n$. Consider the $k$-skeletons $S_k(x_i, \rho(x_i))$, $i=1,\ldots,u$. We define the sets
\[
E_{\pi_j}:= \{ x_i \in Q^*_z :  \Phi(S_k (x_i,\rho(x_i)))= l^{\rho}_{i,z} \,  \text{is parallel to}\, \pi_j \},
\]
In the case $k=0$, we just consider the whole space $Q^*_z$.
\end{definition}
We will also define $\Psi: Q_z \rightarrow Q_z^*$, $x \rightarrow x^*$. This function maps each one of the $n$-cubes $Q_z(i)$ to its center $x_i$. Therefore it is a constant on each one of the $Q_z(i)$, $i=1,\ldots, u$, and $\Psi^{-1}(E_{\pi_j})$, $j=1,\ldots, \binom{n}{k}$, is a partition of $Q_z$. 

With this definition we can give the following generalized version of Proposition \ref{pro:normaindicadoras}
 \begin{proposition}
Let $0<\delta <1$, $z \in \Z^n$,  $\rho\in \Gamma$ and consider the $n$-cubes $Q_z(i)$, $1\le i\le v$. Fix $1<p<\infty$ and a weight $w$ on $Q_z$. Now, let $E$ be one of the sets from Definition \ref{def:E_pi sets}. We consider only the $k$-faces $\ell^{\rho}_{i,z}= \ell_i$ such that $x_i \in E$ ($z,\rho$ are fixed) and denote $|E|=v$. 
Let 
$\Psi^{-1}(E)=\bigcup_{j=1}^v Q_z(i_j)$.

 Suppose that there exists some constant $0<K<\infty$ such that 
\begin{equation*}
\norm{\sum_{i=1}^v t_i\1_{\ell_{i,k}}}_{L^{p'} (7 Q_z , w^{1-p'})} \leq K.
\end{equation*}
Then we have
\[
 \norm{\widetilde{M}^k_{\rho,\delta}f}_{L^p(\Psi^{-1}(E), w)} \leq K  \norm{f}_{L^p(7 Q_z, w)}.
\]
\end{proposition}

Finally, we give the analogous results for Theorem \ref{thm:sufficient-local} and Theorem \ref{thm:sufficient}
 \begin{theorem}
 Let $w$ be a weight in $\mathbb{R}^n$. If $w\in A^{S_k}_{p,\rho,z}$ and $p'\in \mathbb{N}$, then 
\begin{equation*}
\|\widetilde{M}^k_{\rho,\delta}f\|_{L^{p}(Q_z,w)} \leq  C (k,n)  \delta^{\frac{1}{p}\left(\frac{(n-k)(2n-1)}{2n}-n\right)}  [w]^{\frac{1}{p}}_{p,\rho,z} \norm{f}_{L^{p}(7Q_z,w)},
\end{equation*}
where $C(k,n)$ is a constant depending on $k,n$.

 \end{theorem}
 
 \begin{theorem}
 Let $w\in A^{S_k}_p$ in $\mathbb{R}^n$ for $p$ such that $p'\in \mathbb{N}$ or $p=1$. For any fixed $\delta>0$, we have that
\begin{equation*}\label{eq:LpBoundM-Rn}
\|M_\delta f\|_{L^p(w)}\le C(k,n)  \delta^{\frac{1}{p}\left(\frac{(n-k)(2n-1)}{2n}-n\right)} [w]^{\frac{1}{p}}_{A^{S_k}_p}\|f\|_{L^p(w)}.
\end{equation*}
 \end{theorem}

\section*{Acknowledgments}
We would like to thank Pablo Shmerkin for bringing this problem into our attention and for many valuables discussions.

\bibliographystyle{amsalpha}

\end{document}